\documentclass{article}

\usepackage{a4wide}
\usepackage{cite}
\usepackage[utf8]{inputenc}
\usepackage[T1]{fontenc}

\usepackage[colorlinks=true, pdfstartview=FitV, linkcolor=blue,citecolor=blue, urlcolor=blue]{hyperref}

\usepackage[french,english]{babel}
\usepackage{amsmath}
\usepackage{amscd}
\usepackage{amssymb}
\usepackage{amsrefs}
\usepackage{amsthm}
\usepackage{cases}
\usepackage{framed}
\usepackage{graphicx}
\usepackage{latexsym}
\usepackage{color}
\usepackage{array,multirow,makecell}
\usepackage[capitalise]{cleveref}

\usepackage{enumerate}

\usepackage{bbm}

\setcellgapes{1pt}
\makegapedcells

\theoremstyle{plain} 
\newtheorem{thm}{Theorem}[section]
\newtheorem{prop}[thm]{Proposition}
\newtheorem{lem}[thm]{Lemma}

\theoremstyle{definition}

\theoremstyle{remark} 
\newtheorem{rmk}[thm]{Remark}

\newcommand{\R}{\mathbb{R}}

\newcommand{\ep}{\varepsilon}

\newcommand{\p}{\partial}
\newcommand{\be}{\begin{equation}}
\newcommand{\ee}{\end{equation}}
\newcommand{\ba}{\begin{aligned}}
	\newcommand{\ea}{\end{aligned}}

\newcommand{\cLep}{\mathcal L_\ep}

\newcommand{\cX}{\mathcal X}
\newcommand{\ue}{u_\ep}
\newcommand{\ve}{v_\ep}
\newcommand{\we}{w_\ep}
\newcommand{\bv}{\bar v} 
\newcommand{\bu}{\bar u} 
\newcommand{\bw}{\bar w} 
\newcommand{\bp}{\bar p} 

\newcommand{\tv}{v}
\newcommand{\tu}{u} 
\newcommand{\tw}{w} 
\newcommand{\tx}{\widetilde{x}}
\newcommand{\ty}{\widetilde{y}}

\DeclareFontFamily{OT1}{pzc}{}
\DeclareFontShape{OT1}{pzc}{m}{it}{<-> s * [1.10] pzcmi7t}{}
\DeclareMathAlphabet{\mathpzc}{OT1}{pzc}{m}{it}

\title{Partially congested propagation fronts in one-dimensional Navier-Stokes equations}
\author{Anne-Laure Dalibard\footnote{Sorbonne Université, Université Paris-Diderot SPC, CNRS,  Laboratoire Jacques-Louis Lions, LJLL, F-75005 Paris; dalibard@ann.jussieu.fr} \ and Charlotte Perrin\footnote{Aix Marseille Univ, CNRS, Centrale Marseille, I2M, Marseille, France; charlotte.perrin@univ-amu.fr}}

\begin{document}

\maketitle

\begin{small}

\begin{abstract}
These notes are dedicated to the analysis of the one-dimensional free-congested Navier-Stokes equations. 
After a brief synthesis of the results obtained in~\cite{DP} related to the existence and the asymptotic stability of partially congested profiles associated to the \emph{soft congestion} Navier-Stokes system, we present a first local well-posedness result for the one-dimensional free-congested Navier-Stokes equations.

\end{abstract}

	
\end{small}


\section{Introduction}{\label{sec:intro}}
In these notes, we are interested in the following one-dimensional Navier-Stokes system written in Lagrangian coordinates
\begin{subnumcases}{\label{eq:NS-lim}}
\partial_t v - \partial_x u = 0, \label{eq:NS-lim-v}\\
\partial_t u + \partial_x p - \mu \partial_x \left(\dfrac{1}{v}\partial_x u\right) = 0,
\end{subnumcases}
complemented with the \emph{unilateral constraint}
\begin{equation}\label{eq:constraint}
v \geq 1, ~ (v-1) p= 0, ~ p \geq 0.    
\end{equation}
The variable $v$ denotes the specific volume of the fluid and is forced to be greater than $1$, $u$ represents the velocity and $p$ is the pressure associated to the specific volume constraint.
Eventually, the constant $\mu > 0$ represents the viscosity of the fluid.

The analysis of system~\eqref{eq:NS-lim}-\eqref{eq:constraint} is motivated by the modeling of partially \emph{congested} (or saturated) flows like crowd motions or traffic flows \cites{berthelin2017,degond2011,maury2011} or mixtures \cite{bouchut2000} where the constraint $v \geq 1$ can be assimilated to a maximal packing or volume fraction constraint. 
A similar formulation to~\eqref{eq:NS-lim}-\eqref{eq:constraint} was also recently derived for the modeling of partially pressurized free surface flows \cites{lannes2017,godlewski2018}.\\
As developed in previous studies (see for instance~\cite{perrin2015},~\cite{bianchini2020}), equations~\eqref{eq:NS-lim}-\eqref{eq:constraint} can be approximated by the compressible Navier-Stokes equations
\begin{subnumcases}{\label{eq:NS-ep}}
\partial_t v_\ep - \partial_x u_\ep = 0, \\
\partial_t u_\ep + \partial_x p_\ep(v_\ep) - \mu \partial_x \left(\dfrac{1}{v_\ep}\partial_x u_\ep\right) = 0, \\
p_\ep(v) \underset{v \to 1^+}{\to} + \infty.
\end{subnumcases}
with a singular pressure law $p_\ep$ representing repulsive forces that prevent the development of congested phases.
In the previously cited papers, the rigorous justification of the limit $\ep \to 0$ yields the existence of global weak finite energy solutions to \eqref{eq:NS-lim}-\eqref{eq:constraint} (we do not detail the precise setting of these results and refer to~\cite{perrin2015}-\cite{bianchini2020}).
The analysis reveals the multi-scale nature of problem~\eqref{eq:NS-ep}: given a pressure law, {\it e.g.} $p_\ep(v) = \ep (v-1)^{-\gamma}$, one observes that small variations in the specific volume variable (of order $\ep^{1/\gamma}$ here) lead to large variations of the pressure in the highly dense regions where $v$ close to $1$.
Besides, one shows that $\partial_x u = 0$ (or $\mathrm{div}\ u= 0$ in the multi-dimensional case) where $\{v=1\}$ on the limit system~\eqref{eq:NS-lim}-\eqref{eq:constraint}.
Therefore system~\eqref{eq:NS-lim}-\eqref{eq:constraint} can be interpreted as a compressible-incompressible free boundary problem with an interface depending on the solution itself.
This interface is moreover not closed, {\it i.e.} matter passes through the boundary.
Up to our knowledge, few results are known in the literature regarding the existence of more regular solutions to~\eqref{eq:NS-lim}-\eqref{eq:constraint} and the dynamics of the congested domain as time evolves.

\medskip
In these notes we shall focus on particular partially congested solutions that have a stationary profile for both systems~\eqref{eq:NS-lim}-\eqref{eq:constraint} and~\eqref{eq:NS-ep}: $(v,u)(t,x) = (\tv,\tu)(\xi)$ with $\xi = x- st$ and $s$ is the constant speed at which the profile travels. 
We show below that these profiles,
denoted in the following $(\bar v, \bar u)$ and $(\bar v_\ep,\bar u_\ep)$ respectively,
give us precious information about the transition from a congested state to a free state. 
Next, we are interested in the existence of regular solutions to~\eqref{eq:NS-ep} and~\eqref{eq:NS-lim}-\eqref{eq:constraint} for perturbations of the profiles $(\bar v_\ep, \bar u_\ep)$ and $(\bar v,\bar u)$ respectively.
In the approximate case, with $\ep > 0$ fixed, we present a global existence and stability result for small (quantified in terms of $\ep$) regular perturbations of $(\bar v_\ep, \bar u_\ep)$.
This result was initially proved in~\cite{DP}.
Concerning the limit system~\eqref{eq:NS-lim}-\eqref{eq:constraint}, we announce a first local well-posedness result for ``compatible'' initial perturbations (not necessarily small) of the limit profile $(\bar v, \bar u)$, whose complete proof will be given in a forthcoming paper.\\
These two results both rely on high order energy estimates satisfied by regular solutions of the two systems.
In both cases, the derivation of such estimates is facilitated by a change of velocity variable.
Introducing the \emph{effective velocity} $w := u - \mu \partial_x \ln v$, the mass equation rewrites as
\[
\partial_t v - \partial_x w - \mu \partial^2_x \ln v = 0,
\]
with an additional (nonlinear) diffusion term compared to~\eqref{eq:NS-lim-v} which let us expect regularization effects on the specific volume variable $v$ (see also~\cite{vasseur2016}).
Nonetheless, the geometries of the two settings strongly differ from one another.
In the approximate case ($\ep > 0$) the system~\eqref{eq:NS-ep} is set on $\R$. 
On the limit system \eqref{eq:NS-lim}-\eqref{eq:constraint}, we shall restrict ourselves to initial perturbations localized on $\R_+$, {\it i.e.} in the free domain of $(\bar v, \bar u)$. Looking for solutions that remain partially congested, the system~\eqref{eq:NS-lim}-\eqref{eq:constraint} is then studied only on the half-line $[\tilde{x}(t),+\infty[$ where $\tilde{x}(t)$ denotes the position of the interface between the free and the congested domain at time $t$.
Therefore additional difficulties in that case are expected due the free boundary $x=\tilde{x}(t)$ which is an unknown of the system.

\bigskip
The notes are organized as follows. 
First, in Section~\ref{sec:profiles}, we prove the existence and give qualitative properties of partially congested propagation fronts for both systems~\eqref{eq:NS-lim}-\eqref{eq:constraint} and~\eqref{eq:NS-ep}.
Then, we present in Section~\ref{sec:stab-ep} a result concerning the asymptotic stability of the approximate profiles $(\bar v_\ep,\bar u_\ep)$.
Finally, in Section~\ref{sec:LWP}, we announce a local well-posedness result on the limit system~\eqref{eq:NS-lim}.


\section{Partially congested propagation fronts and their soft congestion approximation}{\label{sec:profiles}}

In this section we construct traveling wave profiles for both limit and approximated problems.
For that purpose, we complement system~\eqref{eq:NS-lim}-\eqref{eq:constraint} with the conditions
\begin{equation}\label{eq:endstates}
(v,u)(t,x) \to (v_\pm,u_\pm) \quad \text{as}~ x \to \pm \infty,
\end{equation}
and we assume that 
\begin{equation}\label{eq:endstates-0}
v_- = 1 < v_+, \qquad u_- > u_+,
\end{equation}
in other words we assume that the left end state is congested while the right end state is free.
The condition  $u_+ < u_-$ is an entropy condition on the shock (we shall only consider positive speed shocks).

\subsection{Explicit free-congested front for the limit system}

\begin{lem}\label{lem:profile-ref}
Assume that $u_->u_+$, $v_+>1$, and let
\begin{equation}\label{df:speed}
s:=\frac{u_--u_+}{v_+-1}.
\end{equation}
Then there exists a unique (up to a shift) travelling wave solution of \eqref{eq:NS-lim}. This travelling wave propagates at speed $s$ and is of the form $(\bu, \bv)(x-st)$. Furthermore, 
\[
 \bv(x)= \begin{cases}
1&\text{ if }x\leq 0,\\
\displaystyle\frac{v_+}{1 + (v_+-1) \exp(-sv_+x/\mu)}&\text{ if }x> 0,
\end{cases}.
\]
\[
\bu=u_+ + s v_+ - s\bv = u_- + s v_- - s\bv.
\]
 In the zone $x<0$, the pressure is constant and equal to $p_-= s^2 (v_+ - 1)$. 

\label{lem:TW}
\end{lem}

The profile is represented in Figure~\ref{fig:profile}.
\begin{figure}[h]
\begin{center}
\includegraphics[scale=0.35]{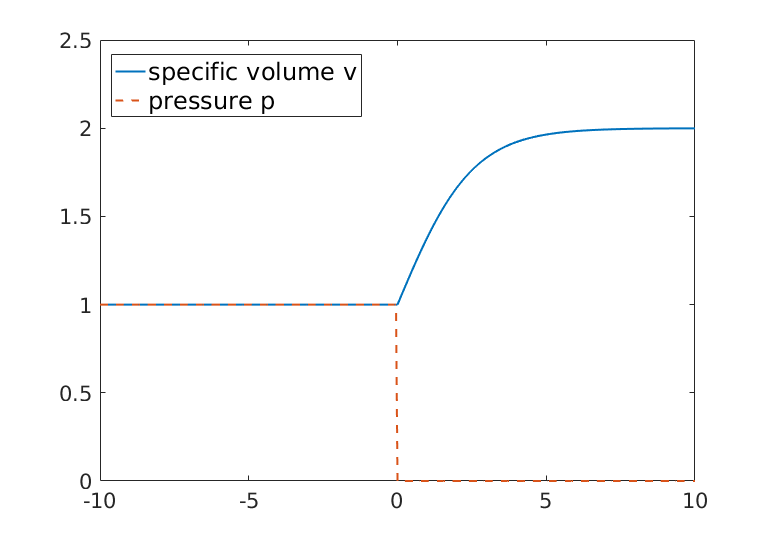}
\includegraphics[scale=0.35]{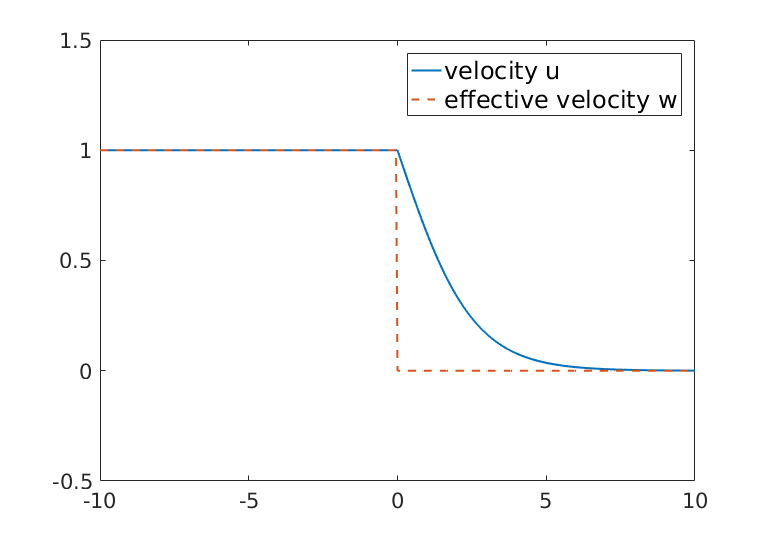}
\end{center}
\caption{Let $v_+ = 2,\ u_- = 1, \ u_+ = 0$. On the left: the profiles $\bv$ and $\bp$, on the right: the profiles $\bu$ and $\bw = \bu - \mu \p_x \ln \bv$. }
\label{fig:profile}
\end{figure}

\begin{proof}[Sketch of proof]
Let $s$ be the speed of propagation of the profile $(\bar v, \bar u)$. 
System~\eqref{eq:NS-lim} becomes
\begin{equation}\label{eq:EDO-prof}
\begin{cases}
-s \bar v' - \bar u' = 0, \\
-s \bar u' + \bar p' - \mu \left(\dfrac{\bar u'}{\bar v}\right)' = 0.
\end{cases}
\end{equation}
We look for a profile which is congested on $]-\infty, 0]$ and free on $]0,+\infty[$\footnote{It can be proved that the analysis can always be reduced to this case, see \cite{DP}.}.
In the free zone, it is easy to show that the dynamics reduces to a logistic equation for both variables:
\[
\begin{cases}
\bv' = \dfrac{s}{\mu} \bv (v_+ -\bv), \\
\bu' = -s \bv' = \dfrac{1}{\mu} (u_+-\bu)(sv_++u_+- \bu),
\end{cases}
\]
while in the congested zone we have $\bv=1$, $\bu =u_-$.
Next, the two dynamics are connected at the point $x^*=0$ by imposing the continuity of $\bv$, $\bu$ and the flux $\bp - \mu \frac{\bu'}{\bv}$.
We recover then the value of the shock speed~\eqref{df:speed} be integrating~\eqref{eq:EDO-prof} between $0$ and $+\infty$.
Eventually, the value of the pressure in the congested domain $x < 0$ is given by: $\bp(x) = - \mu \lim_{x \to 0^+} \bu'(x) = \mu s \lim_{x \to 0^+} \bv'(x) = s^2 (v_+-1)$.
\end{proof}

\subsection{Approximation through the soft congestion approach}
We are now interested in travelling wave profiles $(\bv_\ep,\bu_\ep)$ associated to the soft congestion problem. 
To fix the ideas and make the analysis more explicit, we specify the singular pressure law and set
\begin{equation}\label{df:pep}
p_\ep(v) = \dfrac{\ep}{(v-1)^\gamma} \quad \text{for}~ v > 1, ~\text{with}~ \gamma \geq 1.
\end{equation}
As the reader may check, the results presented in the rest of the notes can be generalized to other pressure laws that are strictly decreasing, convex on $]1, +\infty[$ and singular close to $1$. 
Let $s_\ep$ be the speed of propagation of the profile, we have the following system of ODEs
\[
\begin{cases}
s_\ep \bv'_\ep + \bu'_\ep = 0,  \\
-s_\ep \bu'_\ep + (p_\ep(\bv_\ep))' - \left(\dfrac{\bu'_\ep}{\bv_\ep}\right)' = 0.
\end{cases}
\]
supplemented with the far field conditions
\be\ba
\label{eq:endstates-ep}
\ve(t,x)\to v_\pm^\ep \quad \ue(t,x)\to u_\pm \quad \text{as }x\to \pm \infty,\\
1<v_-^\ep<v_+,\quad u_->u_+.
\ea
\ee
Note that the left limit condition on the specific volume has to be modified in view of the pressure law~\eqref{df:pep} so that $p_\ep(v_-^\ep)$ remains bounded as $\ep \to 0$.
We then set $ v_-^\ep := 1 + \ep^{1/\gamma}$, and we take $v^\ep_+=v_+ $ independent of $\ep$.

\begin{lem}\label{lem:profile-ep}
Assume that $u_- > u_+$, $v_+ > v_-^\ep$ and let 
\[
s_\ep := \sqrt{-\dfrac{p_\ep(v_+) -1}{v_+-1}}.
\]
Then there exists a unique (up to a shift) traveling wave solution $(\bv_\ep,\bu_\ep)$ solution of~\eqref{eq:NS-ep} with end states $(v_-^\ep,u_-)$ (resp. $(v_+,u_+)$) at $-\infty$ (resp. $+\infty$).\\
Moreover, fixing the shift by setting $\bv_\ep(0) = 1 + \ep^{1/(\gamma+1)}$ and taking $ v_-^\ep := 1 + \ep^{1/\gamma}$, we have 
, up to a subsequence,
\[
\bv_\ep \to \bv \quad \text{in}~\mathcal{C}(-R,R) ~\forall R > 0~ \text{and weakly-* in}~ W^{1,\infty}(\R).
\]
\end{lem}

\bigskip
\begin{proof}[Sketch of proof]
It can be shown that 
\be\label{EDO-vep}
\bv'_\ep = \dfrac{\bv_\ep}{\mu s_\ep} \left(s_\ep^2(v_+ -\bv_\ep) + p_\ep(v_+) - p_\ep(\bv_\ep)\right),
\ee
so that existence and uniqueness (up to a shift) of a monotone (increasing) profile $\bv_\ep$ follows easily from ODEs arguments and the convexity of the pressure law $p_\ep$. 
Observing that $p_\ep(\bv_\ep) \leq p_\ep(v_-^\ep) = 1$, we control $\bv'_\ep$ uniformly with respect to $\ep$ and infer the uniform convergence of $\bv_\ep$ towards $\bv$.
\end{proof}

\bigskip

We can actually be more precise on the behavior of $\bv_\ep$.
We distinguish between three zones:
\begin{itemize}
\item {\it Congested zone:} this corresponds to a zone $]-\infty, x_{min}]$ in which $\bar v_\ep=1 + O(\ep^{1/\gamma})$, so that the pressure remains bounded. 
In this zone, the analysis of the linearized version of \eqref{EDO-vep} around $v^\ep_-$ shows that the profile $\bv_\ep$ converges exponentially towards $v_-^\ep$.\\
Computations show that $x_{min} = O(\ep^{1/(\gamma+1)})$.

\item {\it Free zone:} this corresponds to the region where $p_\ep (\bar v_\ep)\ll 1$, $p_\ep' (\bar v_\ep)=O (1)$.
With the choice of the shift above, this corresponds to the interval $[0, +\infty[$. In this zone, one proves that $\|\bv_\ep - \bv\|_{L^\infty} \lesssim \ep^{\frac{1}{\gamma+1}}$.

\item {\it Transition zone $[-x_{min},0]$:} in this small region, we have the following error estimate.

\begin{lem}
Assume the conditions of the previous lemma. Let $\tilde{v}$ be the solution of the ODE
\[
\tilde{v}' = \dfrac{1}{\mu s}\left(1-\dfrac{1}{\tilde{v}^{\gamma}} \right), \qquad \tilde{v}(0) = 2,
\]
and let $x^* < 0$ be a suitable parameter such that $x^* = O(\ep^{1/(\gamma+1)})$.
Then 
\[
\Big|
\bv_\ep(x) - \bv(x) - \ep^{\frac{1}{\gamma}} \tilde v \left(\dfrac{x - x^*}{\ep^{1/\gamma}}\right)
\Big|
\leq C \ep^{\frac{1}{\gamma+1}} |x| 
\qquad \forall \ x \in [-x_{min},0],
\]
with $x_{min} = O(\ep^{1/(\gamma+1)})$.
\end{lem}

\end{itemize}

\section{Perturbations of the approximate profiles $(\bar v_\ep,\bar u_\ep)$}{\label{sec:stab-ep}}

As said in the introduction, it will be convenient to rewrite Eq~\eqref{eq:NS-ep} in the variables $(v,w)$, that is
\[
\begin{cases}
\partial_t v - \partial_x w - \mu \partial^2_x \ln v = 0 \\
\partial_t w + \partial_x p_\ep(v) = 0
\end{cases}
\quad \text{for}~ t > 0,\  x \in \R,
\]
with 
\[
(v,w) \to (v^\ep_\pm, u_\pm) \quad \text{as}~ x \to \pm \infty,
\]
and initial data
\[
(v,w)_{t=0} = (v^0, u^0 - \mu \partial_x \ln v^0).
\]
We recall that $v_+^\ep:=v_+$ and $u_\pm$ are independent of $\ep$, and that $v^\ep_-=1 + \ep^{1/\gamma}$.
Now, assuming that $(v,w)_{|t=0} \in (\bar v_\ep, \bar w_\ep) + L^1_0 \cap L^\infty(\R)$ where $L^1_0(\R)$ is the set of $L^1$ functions of zero mass, we look at the system
\begin{subnumcases}{\label{eq:NS-ep-int}}
\partial_t V - \partial_x W - \mu \partial_x \ln\left(1+ \dfrac{\partial_x V}{\bar v_\ep(x-s_\ep t)} \right) = 0 \\
\partial_t W + p_\ep(\bar v_\ep(x-s_\ep t) + \partial_x V) - p_\ep(\bar v_\ep(x-s_\ep t)) = 0
\end{subnumcases}
satisfied by the integrated variables
\[
V(t,x) := \int_{-\infty}^x (v (t,z)- \bar v_\ep(z-s_\ep t)) \ dz, \qquad
W(t,x) := \int_{-\infty}^x (w (t,z)- \bar w_\ep(z-s_\ep t))  \ dz.
\]
Our strategy is the following: first we prove by a fixed point argument the existence and uniqueness a global regular solution $(V,W)$ under smallness assumptions on $(V,W)_{|t=0}$; then we come back to the original variables and deduce the existence and uniqueness of a couple $(v,u)$ solution to~\eqref{eq:NS-ep}.
The regularity of $(v-\bar v_\ep (\cdot-s_\ep t),u-\bar u_\ep(\cdot-s_\ep t))$ will eventually ensure the asymptotic stability of $(\bar v_\ep,\bar u_\ep)$, {\it i.e.} the convergence of $(v-\bar v_\ep (\cdot-s_\ep t),u-\bar u_\ep(\cdot-s_\ep t))$ to $0$ as $t\to +\infty$.

In the whole subsection, we fix
\[
p_\ep(v) := \dfrac{\ep}{(v-1)^\gamma} \quad \text{for}~ v >1, \ \ep > 0, \ \gamma \geq 1. 
\]

\medskip
Here is the first result guaranteeing the existence of a global strong solution $(V,W)$.
\begin{prop}\label{prop:ex-VW}
	Assume that $(V_0,W_0) \in (H^2(\R))^2$ with
	\begin{equation}\label{eq:init-1}
	\sum_{k=0}^2 \ep^{\frac{2k}{\gamma}}\int_{\R}{\left[ \dfrac{|\partial^k_x W_0|^2}{-p'_\ep(\bar v_\ep)} + |\partial^k_x V_0|^2 \right] dx} \leq \delta_0^2 \ep^{\frac{5}{\gamma}}
	\end{equation}
	for some $\delta_0$ small enough, depending only on $v_+$, $\gamma$ and $\mu$.
	Then there exists a unique global solution $(V,W)$ to~\eqref{eq:NS-ep-int} satisfying
	\begin{align*}
	& V \in \mathcal{C}([0;+\infty);H^2(\R)) \cap L^2(\R_+; H^3(\R)), \\
	& W \in \mathcal{C}([0;+\infty);H^2(\R)).
	\end{align*}
\end{prop}

\medskip
\begin{proof}[Idea of the proof]
As announced previously, this result is achieved thanks to a fixed point argument which relies on energy estimates satisfied by $(V,W)$ and its derivatives $(\partial^k_x V, \partial^k_x W)$, $k=1,2$.
Since we are working close to the reference profile $(\bar v_\ep(\cdot-s_\ep t), \bar{w}_\ep(\cdot-s_\ep t))$, it is natural to rewrite the system~\eqref{eq:NS-ep-int} as follows
\begin{equation}\label{eq:linearized-VW}
\partial_t \begin{pmatrix} V \\ W \end{pmatrix}  + \cLep \begin{pmatrix} V \\ W \end{pmatrix} = \begin{pmatrix} F_\ep \\ G_\ep \end{pmatrix},
\end{equation}
with a linear left-hand side
\[
\cLep\begin{pmatrix} V \\ W \end{pmatrix} := 
\begin{pmatrix} 
- \partial_x W - \mu \partial_x \left(\dfrac{\partial_x V}{\bar v_\ep(\cdot-s_\ep t)} \right) \\
p'_\ep(\bar v_\ep(\cdot-s_\ep t)) \partial_x V
\end{pmatrix}
\]
that yields the main order part of the energy and dissipation terms; and a right-hand side
\begin{align*}
	F_\ep & = F_\ep(\partial_x V) := \mu \p_x \left[ \ln \left(1 + \dfrac{\partial_x V}{\bar v_\ep(\cdot-s_\ep t)}\right) - \dfrac{\partial_x V}{\bar v_\ep(\cdot-s_\ep t)} \right],\\
	G_\ep & = G_\ep(\partial_x V): = - \left[p_\ep (\bar v_\ep (\cdot-s_\ep t) + \partial_x V) - p_\ep (\bar v_\ep(\cdot-s_\ep t)) - p_\ep'(\bar v_\ep(\cdot-s_\ep t)) \partial_x V\right],
\end{align*}
which is quadratic in $\p_x V$ and will be treated as a perturbation.
Hence, taking the scalar product of \eqref{eq:linearized-VW} with $(V, \frac{W}{-p'_\ep(\bar v_\ep(\cdot-s_\ep t))})$, we get the energy estimate
\begin{align}\label{eq:k=0}
 \int_\R{\left[-\dfrac{1}{p'_\ep(\bar v_\ep)}{|W|^2} + {|V|^2} \right]}
	+ s_\ep\int_{\R_+}\int_\R{\dfrac{p''_\ep(\bar v_\ep)}{(p'_\ep(\bar v_\ep))^2} \partial_x \bar v_\ep |W|^2}
	+ 2\mu \int_{\R_+}\int_\R{\dfrac{(\partial_x V)^2}{\bar v_\ep}} & \nonumber\\
 \leq \int_\R{\left[\left(\dfrac{-1}{p'_\ep(\bar v_\ep)}{|W|^2}\right)_{|t=0} + {|V_{|t=0}|^2} \right]} 
	+ 2\left|\int_{\R_+}\int_\R \left[G_\ep\frac{W}{-p_\ep'(\bar v_\ep)} + F_\ep V\right]\right|,
\end{align}
where we have abusively written $\bar v_\ep$ as a short-hand for $\bar v_\ep(\cdot-s_\ep t)$ and where the integrals involving $F_\ep$ and $G_\ep$ are controlled by assuming that the distance between $(v,w)$ and the profile $(\bar v_\ep, \bar w_\ep)$ remains small enough.
Regarding the diffusion term on $W$, note that 
\[
\frac{p_\ep''(\bar v_\ep)}{p_\ep'(\bar v_\ep)^2}= \frac{(\gamma+1) (\ve-1)^\gamma}{\gamma\ep }= \frac{\gamma+1}{\gamma p_\ep(\bar v_\ep)}\geq \frac{\gamma+1}{\gamma}.
\]
To derive higher order estimates, we differentiate the system with respect to $x$ and perform the same calculations. 
Nevertheless, we need to take into account additional terms coming from the commutator of $\cLep$ and $\partial_x^k$, $k=1,2$. 
For $k=1$, we have
\begin{align}\label{eq:k=1}
& \int_\R{\left[-\dfrac{1}{p'_\ep(\bar v_\ep)}{|\partial_x W|^2} + {|\partial_x V|^2} \right]}
+ s_\ep\int_{\R_+}\int_\R{\dfrac{p''_\ep(\bar v_\ep)}{(p'_\ep(\bar v_\ep))^2} \partial_x \bar v_\ep |\partial_x W|^2}
+ 2\mu \int_{\R_+}\int_\R{\dfrac{(\partial^2_x V)^2}{\bar v_\ep}}  \nonumber\\
& \leq \int_\R{\left[\left(\dfrac{-1}{p'_\ep(\bar v_\ep)}{|\partial_x W|^2}\right)_{|t=0} + {|\partial_x V_{|t=0}|^2} \right]} 
+ 2\left|\int_{\R_+}\int_\R \left[\partial_x G_\ep\frac{\partial_x W}{-p_\ep'(\bar v_\ep)} + \partial_x F_\ep \partial_x V\right]\right| \\
& \quad + \left|\int_{\R_+}\int_{\R}{[\cLep, \p_x]\begin{pmatrix}V\\ W\end{pmatrix} \cdot \begin{pmatrix}\dfrac{-\partial_x W}{p'_\ep(\bar v_\ep)}\\ \partial_x W\end{pmatrix}} \right| \nonumber
\end{align}
with 
\begin{align*}
\left|\int_{\R_+}\int_{\R}{[\cLep, \p_x]\begin{pmatrix}V\\ W\end{pmatrix} \cdot \begin{pmatrix}\dfrac{-\partial_x W}{p'_\ep(\bar v_\ep)}\\ \partial_x W\end{pmatrix}} \right| 
\leq \eta \int_{\R_+} \int_{\R}\partial_x \bar v_\ep |\partial_x W|^2
+ \frac{C_1}{\eta}\ep^{-2/\gamma} \int_{\R_+} \int_{\R}{|\partial_x V|^2}.
\end{align*}
The first integral can be absorbed in the left-hand side for small $\eta$. 
For the second integral, we could apply a Gronwall inequality to close the estimate but we would then obtain a bound on the energy that exponentially grows with time. 
Another way to proceed is to multiply inequality~\eqref{eq:k=1} by $\ep^{2/\gamma}$ (eliminating the singularity as $\ep \to 0$), and to combine the result with the estimate on the integrated variables~\eqref{eq:k=0}:
\begin{align}\label{eq:k=1-bis} 
& \int_\R{\left[-\dfrac{1}{p'_\ep(\bar v_\ep)}{|W|^2} + {|V|^2} \right] + \ep^{2/\gamma}\int_\R \left[-\dfrac{1}{p'_\ep(\bar v_\ep)}{|\partial_x W|^2} + {|\partial_x V|^2} \right] } \nonumber \\
& \quad + \int_{\R_+}\int_\R{\partial_x \bar v_\ep \left[|W|^2+ \ep^{2/\gamma}|\partial_x W|^2\right]}
+ \int_{\R_+}\int_\R{\dfrac{(\partial_x V)^2+ \ep^{2/\gamma}(\partial^2_x V)^2}{\bar v_\ep}}  \nonumber\\
& \leq C \int_\R{\left[\left(\dfrac{-1}{p'_\ep(\bar v_\ep)}{|W|^2}\right)_{|t=0} + {|V_{|t=0}|^2} + \ep^{2/\gamma}\left(\left(\dfrac{-1}{p'_\ep(\bar v_\ep)}{|\partial_x W|^2}\right)_{|t=0} + {|\partial_x V_{|t=0}|^2} \right)\right]} \\
& \quad  + C\left| \int_{\R_+}\int_\R \left[G_\ep\frac{W}{-p_\ep'(\bar v_\ep)} + F_\ep \ V\right] \right|
+ C \ep^{2/\gamma}\left|\int_{\R_+}\int_\R \left[\partial_x G_\ep\frac{\partial_x W}{-p_\ep'(\bar v_\ep)} + \partial_x F_\ep \ \partial_x V\right]\right|. \nonumber
\end{align}
The passage to the integrated variables $(V,W)$ is therefore essential for the derivation of global-in-time estimates.\\
In the same manner, for $k=2$, the reader can check that by multiplying by $\ep^{4/\gamma}$ the energy inequality satisfied by $(\partial^2_x V,\partial^2_x W)$ and combining it with~\eqref{eq:k=1-bis}, we can close a weighted energy estimate (recall that the nonlinear terms in $F_\ep$, $G_\ep$ are considered as small perturbations).\\
This explains the structure of the energy in assumption \eqref{eq:init-1}. 
Hence, we define
\[\ba 
E_k(t;V,W):= \int_\R{\left[\dfrac{-1}{p'_\ep(\bar v_\ep(\cdot -s_\ep t)}{|\p_x^k W (t)|^2}+ {|\p_x^k V(t)|^2} \right] dx},\\
D_k(t;V,W):= \int_\R \partial_x\bar v_\ep (\cdot -s_\ep t)|\p_x^k W|^2  dx
+  \int_\R{{(\partial_x^{k+1} V)^2} dx}.
\ea
\]
%
The goal is to prove, by a fixed point argument, existence and uniqueness of global smooth solutions $(V,W)$, under the assumption that $E_k(0)$ is small enough for $k=0,1,2$. 
Given the couple $(V_1,W_1)$, we introduce the following system
\begin{align*}\label{eq:syst-fixed-point}
&\p_t \begin{pmatrix} V_2\\ W_2 \end{pmatrix} + \cLep\begin{pmatrix} V_2\\ W_2 \end{pmatrix}
= \begin{pmatrix} F_\ep(\p_x V_1)\\ G_\ep(\p_x V_1) \end{pmatrix}\\
&(V_2, W_2)_{|t=0}=(V_0, W_0)
\end{align*}
and the application
\[
\mathcal A^\ep :( V_1, W_1) \in \cX \mapsto ( V_2, W_2) \in \cX,
\]
where
\[
\cX:= \{ (V,W) \in L^\infty(\R_+;  H^2(\R))^2 ;\ D_k(t;W,V)\in L^1(\R_+) \text{ for }k=0,1,2 \}.
\]
We endow $\cX$ with the norm
\begin{equation}\label{eq:df-norm}
\|(V,W)\|_{\cX}^2:= \sup_{t\in [0,+\infty[} \left[\sum_{k=0}^2 c^k \ep^{2k/\gamma} \left[E_k(t,V(t),W(t)) + \int_0^t D_k(s,V(s),W(s))\:ds\right]
\right],
\end{equation}
where $c$ will be taken small but independent of $\ep$.
For $\delta>0$, we denote by $B_\delta$ the ball 
\begin{equation}
B_\delta = \{(V,W)\in \cX,\ \|(V,W)\|_{\cX}< \delta \ep^{\frac{5}{2\gamma}}\}.
\end{equation}
Under the assumptions of Proposition~\ref{prop:ex-VW}, there exists $\delta = \delta(\delta_0,v_+,\mu,\gamma)$ such that the ball $B_\delta$ is stable by $\mathcal{A}_\ep$ and such that furthermore $\mathcal{A}_\ep$ is a contraction on $B_\delta$. See details of the proof in~\cite{DP}.
\end{proof}

\bigskip
Let us now return to the original variables $(v=\bar v_\ep(\cdot -s_\ep t) + \partial_x V,u)$. In the rest of this paper, we will write $(\bar v_\ep, \bar u_\ep)$ for $(\bar v_\ep, \bar u_\ep)(\cdot -s_\ep t)$ in order to lighten the notation. 

\medskip
\begin{lem}\label{lem:stability-u}
	Assume that initially $(U_0,V_0)\in H^2(\R)\times H^3(\R)$ is such that~\eqref{eq:init-1} is satisfied by the couple $(W_0,V_0)$
	and consider the solution $(W,V) \in B_\delta \subset \cX$ of~\eqref{eq:NS-ep-int} given by the previous proposition.
	Then there exists a unique regular solution $u$ to
	\begin{align}\label{eq:upertub}
    \partial_t(u -\bar u_\ep) - \mu\partial_x\left(\frac{1}{v}\partial_x(u-\bar u_\ep)\right) 
    & = - \partial_x(p_\ep(v) -p_\ep(\bar v_\ep)) 
    + \mu \partial_x\left(\left(\frac{1}{v}-\frac{1}{\bar v_\ep}\right)\partial_x \bar u_\ep \right),
    \end{align} 
	which is such that
	\begin{align}
	u-\bar u_\ep \in \mathcal{C}([0,+\infty); H^1(\R)) \cap L^2([0,+\infty),H^2(\R)), \quad 
	\partial_t(u-\bar u_\ep) \in L^2([0,+\infty)\times \R).
	\end{align}	
\end{lem}

Let us note that, until now, we did not justify properly the passage to the integrated system~\eqref{eq:NS-ep-int}.
The equivalence between the original system~\eqref{eq:NS-ep} and~\eqref{eq:NS-ep-int} is established by proving $L^1$ bounds on $v-\bar v_\ep$, $u-\bar u_\ep$, $w- \bar w_\ep$.

\begin{lem}
Assume that the initial data $(u_0,v_0)$ is such that
\[
u_0 - \bar u_\ep \in W^{1,1}_0(\R)\cap H^1(\R), \quad
v_0 - \bar v_\ep \in W^{2,1}_0(\R)\cap H^2(\R).
\]
Then for all times $t\geq 0$, $(v-\bar v_\ep)(t,\cdot)$ and $(u-\bar u_\ep)(t,\cdot)$ belong to $L^1_0(\R)$.
\end{lem}

\begin{proof}[Idea of the proof]
We first derive $L^1$ bounds on $u-\bar u_\ep $ and $w-\bar w_\ep$, which follow from ideas from~\cite{haspot2018}.
We consider a sequence $(j_n)_{n\in \mathbb N}$ of $\mathcal C^2$, convex functions, converging as $n\to +\infty$ towards $|\cdot|$ in $W^{1,\infty}$. We multiply the equation on $u-\bar u_\ep$ (resp. on $w-\bar w_\ep$) by $j_n'(u-\bar u_\ep)$ (resp. $j_n'(w-\bar w_\ep)$) and perform integrations by part.
Using the convexity of $j_n$, we observe that the diffusion term has a positive sign.
We obtain eventually
\[
\frac{d}{dt}\int_{\R} \left(j_n(u-\bar u_\ep) + j_n (w-\bar w_\ep)\right) \leq C_\ep\left( 1 + \|u-\bar u_\ep\|_{L^1} +  \|w-\bar w_\ep\|_{L^1} \right).
\]
The constant $C_\ep$ involves bounds on $ \bar u_\ep, \bar v_\ep$ in various Sobolev spaces ($W^{2,1}$, $W^{1,\infty} $),  on $\|w-\bar w_\ep\|_{L^\infty(\R_+, W^{1,\infty}(\R))}$ and on $\|u-\bar u_\ep\|_{L^\infty(\R_+, W^{1,\infty}(\R))}$. Integrating in time, letting $n\to \infty$ and using a Gronwall lemma, we obtain
	\begin{align} \label{eq:L^1-uw}
	\|(u-\ue)(t)\|_{L^1_x} + \|(w-\we)(t)\|_{L^1_x}
	\leq C_\ep \Big[\|u_0-\ue(0)\|_{L^1_x} + \|w_0-\we(0)\|_{L^1_x}\Big] \ e^{C_\ep t}.
	\end{align}
We then derive similar estimates on $v-\bar v_\ep$.

The obtained estimate is local in time and depends on $\ep$ but we are just interested in the fact that $(v-\bar v_\ep)(t,\cdot)$, $(u-\bar u_\ep)(t,\cdot)$ belong to $L^1(\R)$ to legitimate the study of the integrated system~\eqref{eq:NS-ep-int}.
\end{proof}

Finally, we have the following result.
\begin{thm}[Nonlinear asymptotic stability of partially congested profiles]\label{thm:estimates}
	Assume that the initial data $(u_0,v_0)$ is such that
	\[
	u_0 - \bar u_\ep \in W^{1,1}_0(\R)\cap H^1(\R), \quad
	v_0 - \bar v_\ep \in W^{2,1}_0(\R)\cap H^2(\R),
	\]
	and the associated couple $(W_0,V_0) \in H^2\times H^3 (\R)$ satisfies~\eqref{eq:init-1}.
	Then there exists a unique global solution $(u,v)$ to~\eqref{eq:NS-ep} which satisfies
	\begin{align*}
	& u-\bar u_\ep \in \mathcal{C}([0;+\infty);H^1(\R)\cap L^1_0(\R)), \\
	& v-\bar v_\ep \in \mathcal{C}([0;+\infty);H^1(\R)\cap L^1_0(\R)) \cap L^2(\R_+; H^2(\R)),
	\end{align*}	
	and
	\begin{equation}\label{eq:min-v}
	v(t,x) > 1 \quad \text{for all}~ t,x.
	\end{equation}
	Finally
	\begin{equation}
	\sup_{x\in \R} \ \Big| \big((v,u)(t,x) - (\bar v_\ep,\bar u_\ep)(x-s_\ep t)\big)\Big| \underset{t\rightarrow +\infty}{\longrightarrow} 0.
	\end{equation}
\end{thm}

\begin{proof}[Idea of the proof]
It basically remains us to justify the minimal constraint~\eqref{eq:min-v} and the long-time behavior.\\
For the first point, we write that for $(W,V) \in B_{\delta}= \{ (W,V) \in \mathcal{X}, ~ \|(W,V)\|_{\mathcal{X}} \leq \delta \ep^{\frac{5}{2\gamma}} \}$
\begin{align*}
\|v-\bar v_\ep\|_{L^\infty_{t,x}} 
= \|\partial_x V\|_{L^\infty_{t,x}}
& \leq C \|\partial_x V\|_{L^\infty_t L^2_x}^{1/2} \, \|\partial_x^2 V\|_{L^\infty_t L^2_x}^{1/2} \\
& \leq C \ep^{-\frac{3}{2\gamma}} \|(W,V)\|_{\mathcal{X}} \\
& \leq C \delta \ep^{\frac{1}{\gamma}} ,
\end{align*}
so that for $\delta$ small enough
\[
\|v-\bar v_\ep\|_{L^\infty_{t,x}} < \ep^{\frac{1}{\gamma}}.
\]
Recalling that $\bar v_\ep \geq v_-^\ep =1+\ep^{1/\gamma}$, we deduce that $v(t,x)>1$ for all $t > 0$, $x \in \R$.\\
The asymptotic stability of $(\bar v_\ep,\bar u_\ep)$ easily derives from the regularity of $(v -\bar v_\ep,u-\bar u_\ep)$. 
Indeed, we have on the $v$ variable
\[
v-\ve = \partial_x V \in L^2([0,+\infty);H^2(\R)), \qquad
\partial_t (v-\ve) = \partial_x(u-\ue) \quad \text{in} \quad L^2([0,+\infty);H^1(\R)),
\]
and we infer that
\[
\|(v-\bar v_\ep)(t)\|_{H^1_x} \underset{t\rightarrow +\infty}{\longrightarrow} 0.
\]
As a consequence, we have
\[
|(v-\bar v_\ep)(t,x)| \leq C \|(v -\bar v_\ep)(t)\|_{L^2_x}^{1/2} \|\partial_x(v -\bar v_\ep)(t)\|_{L^2_x}^{1/2}  \underset{t\rightarrow +\infty}{\longrightarrow} 0.
\]
Similar arguments show the uniform convergence of $u -\bar u_\ep$ to $0$ as $t\to +\infty$.
\end{proof}

\section{Local well-posedness result for the free-congested Navier-Stokes equations}
{\label{sec:LWP}}

In this section, we consider the system \eqref{eq:NS-lim}-\eqref{eq:constraint} endowed with an initial data $(u^0, v^0)$. 
To the best of our knowledge, the study of the Cauchy problem for this system has not been tackled before.
Our purpose is similar to the one of the previous section. 
We consider the travelling wave $(\bar v, \bar u)(x-st)$ constructed in Section~\ref{sec:profiles}. 
We start from an initial data which is a perturbation of this profile, and we construct strong local solutions of the system.

However, due to the nature of the system \eqref{eq:NS-lim}-\eqref{eq:constraint}, we will not consider arbitrary perturbations.
We restrict our study to initial data which are perturbations of the profile $(\bar v, \bar u)$ in the non-congested zone only. 
As time evolves, this will allow us to have a more simple description of the non-congested zone, which will simply be a half-line $]\tilde x(t), +\infty[$.
Let us now explicit our assumptions on the initial data~$(\tu^0, \tv^0)$:
\begin{enumerate}[(H1)]
\item{\it  Partially congested initial data: }$(u^0, v^0)\in (\bu, \bv) + L^1(\R)$, and such that $u^0(x)=\bu(x)=u_-$, $v^0(x)=\bv(x)=1$ if $x<0$;

\item {\it Regularity: }$\mathbf 1_{x>0}(\tu^0-\bu, \tv^0-\bu)\in  H^3(\R_+)$; 

\item {\it Compatibility:} $\tu^0(0^+)=u_-$, $\tv^0(0^+)=1$, and
\begin{equation}
\left[-\frac{(\p_x \tu^0)^2}{\p_x \tv^0} - \mu \p_x \tv^0 \p_x u^0 + \mu \p_x^2 \tv^0\right]_{|x=0^+}=0;
\end{equation}

\item{\label{item:no-deg}} {\it Non-degeneracy:}
$\p_x \tv^0(0^+)>0$, $\p_x \tu^0(0^+)>0$ and $\tv^0(x)>1$ for $x>0$;


\end{enumerate}

Under these assumptions, the solution of \eqref{eq:NS-lim}-\eqref{eq:constraint} associated with $(u^0, v^0)$, if it exists, will not be a travelling wave. 
However, it is reasonable to expect such a solution to be congested in a zone $x<\tx(t)$, and non-congested in a zone $x>\tx(t)$, where the free boundary $x=\tx(t)$ is an unknown of the problem. 
We actually announce the following result:

\begin{thm}[Local in time existence and uniqueness]
Let  $(\tu^0, \tv^0)$ satisfying the assumptions (H1)-(H4).
Then there exists $T>0$ and $\tx \in H^2_\text{loc}([0,T[)$, with $\tx(0) = 0$, $\tx'(0) = -[\frac{\partial_x\tu^0}{\partial_x \tv^0}]_{x=0^+}$, such that~\eqref{eq:NS-lim} has a unique maximal solution $(\tu,\tv)$ of the form $(\tu,\tv)(t,x)= (u_s, v_s)(t, x-\tx(t))$ on the interval $[0, T[$, where $u_s(t,x)= u_-$, $v_s(t,x)= 1$ and $p_s(t,x) = -\mu (\partial_x u_s)_{|x=0^+}$ for $x<0$. Furthermore, 
\begin{equation}
v_s(t,x) > 1 \quad \text{for all}~ t\in [0,T[, ~x > 0,
\end{equation}
and the solution $(u_s,v_s)$ has the following regularity in the free domain:
\begin{align}
v_s - \bv, \ u_s -\bu \in L^\infty([0,T[; H^3(\R_+)), \label{reg-uv}\\
\partial_t(v_s -\bv), ~ \partial_t(u_s-\bu) 
\in L^\infty([0,T[;H^1(\R_+)) \cap L^2(]0,T[; H^2(\R_+)).\label{reg-dtuv}
\end{align}
Eventually, the pressure in the congested domain satisfies
\begin{equation}
p_s \in H^1(0,T).
\end{equation}
\label{thm:main-loc}
\end{thm}

The strategy of proof is the following. We work in the shifted variable $x-\tx(t)$. Since $(u,v)$ is expected to be constant in $x-\tx(t)<0$, we only consider the system satisfied by $(u_s, v_s)$ in the positive half-line, which reads
\begin{subnumcases}{\label{eq:NS-shifted-1}}
\partial_t v_s - \tx'(t) \p_x v_s- \partial_x u_s = 0,\quad  t>0,\ x>0\label{eq:NS-shifted-1-v}\\
\partial_t u_s - \tx'(t) \p_x u_s - \mu \partial_x\left(\dfrac{1}{v_s}\partial_x u_s \right) = 0,\quad  t>0,\ x>0\label{eq:NS-shifted-1-u}\\
(v_s,u_s)_{|x=0}= ( 1 ,u_-), \quad \lim_{x\to \infty} (v_s(t,x),u_s(t,x)) = (v_+,u_+) \quad \forall \ t > 0.
\end{subnumcases}
Of course, the dynamics of $\tx$ is coupled with the dynamics of $(v_s, u_s)$. 
In order to construct a solution of \eqref{eq:NS-shifted-1}, it will be more convenient to modify the equation on $v_s$ in order to make the regularizing effects of the diffusion more explicit. Indeed, setting $w_s= u_s - \mu \p_x \ln v_s$, we find that equation \eqref{eq:NS-shifted-1-v} can be written as
\[
\partial_t v_s - \tx'(t) \p_x v_s- \mu\partial^2_x \ln v_s  = \partial_x w_s,\quad t>0,\ x>0.
\] 
Moreover,
\begin{equation}\label{eq:ws}
\p_t w_s - \tx'(t) \p_x w_s = 0,\quad t>0,\ x>0,
\end{equation}
therefore $w_s(t,x)= w^0 (x + \tx(t))$ for all $t>0, x>0$ provided $\tx'(t)>0$ for all $t>0$. In particular, letting $x\to 0^+$, we obtain
\be\label{EDO-1}
u_- - \mu \p_x v_{s|x=0^+}= w^0(\tx(t)).
\ee
Now, taking the trace of  \eqref{eq:NS-shifted-1-v} on $x=0^+$, we have
\be\label{EDO-2}
\tx'(t) \p_x v_{s|x=0^+}= - \p_x u_{s|x=0^+}
.\ee
Gathering \eqref{EDO-1} and \eqref{EDO-2} leads to
\be\label{EDO-tx}
\tx'(t) = -\frac{\p_x u_{s|x=0^+}}{ \p_x v_{s|x=0^+}}= -\mu  \frac{ \p_x u_{s|x=0^+}}{u_- - w^0(\tx(t))},
\ee
which links the dynamics of the interface to $(v_s,u_s)$.

Since $w(t)=w^0(\cdot + \tx(t))$, the equation on $v_s$ rewrites 
\be\label{eq:vs}\ba
\p_t v_s - \tx'(t) \p_x v_s - \mu \p_{x x} \ln v_s= \p_x w^0(x+ \tx(t)),\quad t>0, \ x>0,\\
v_{s|x=0}=1,\quad \lim_{x\to \infty} v_s(t,x)= v_+,\\
v_{s|t=0}=v^0.
\ea\ee
Thus we will build a solution $(\tx, v_s, u_s)$ of \eqref{EDO-tx}-\eqref{eq:vs}-\eqref{eq:NS-shifted-1-u} thanks to the following fixed point argument:
\begin{enumerate}
    \item For any given $\ty\in H^2_\text{loc}(\R_+)$, such that $\ty(0) = 0$ and $\ty'(0)= - \frac{\p_x \tu^0_{|x=0^+}}{\partial_x \tv^0_{|x=0^+}}$, we consider the solution $v$ of the equation
    \be\label{eq:vs-y}\ba
\p_t v - \ty'(t) \p_x v - \mu \p_{x x} \ln v = \p_x w^0(x+ \ty(t)),\quad t>0, \ x>0,\\
v_{|x=0}=1,\quad \lim_{x\to \infty} v(t,x)= v_+,\\
v_{|t=0}=v^0.
\ea\ee
    We prove that under suitable conditions on the initial data, there exists a unique solution $v \in \bv + L^\infty_\text{loc}(\R_+,H^1(\R_+))$, and we derive higher regularity estimates ({\it cf} the regularities announced in~\eqref{reg-uv}-\eqref{reg-dtuv}).
    \item We then consider the unique solution $u\in \bu + L^\infty_\text{loc}(\R_+, H^1(\R_+))$ of
    \be\label{eq:us-y}
    \ba
    \partial_t u - \ty'(t) \p_x u - \mu \partial_x\left(\dfrac{1}{v}\partial_x u\right) = 0\quad t>0, \ x>0,\\
    u_{|x=0}= u_-, \quad \lim_{x\to \infty} u(t,x)= u_+,\\
    u_{|t=0}=u^0,
\ea\ee
    where $v$ is the solution of \eqref{eq:vs-y}. Once again, we derive regularity estimates on $u$ ({\it cf}~\eqref{reg-uv}-\eqref{reg-dtuv}).
    
    \item Eventually, we define
    \[
    \tilde z(t):=-\mu \int_0^t  \  \frac{ \p_x u(\tau, 0)}{u_- - w^0(\tilde y(\tau))}d\tau,
    \]
    and we consider the application $\mathcal A:\tilde y \in H^2_\text{loc}(\R_+) \mapsto \tilde z \in H^2_\text{loc}(\R_+)$.
    
    We prove that  for $T>0$ small enough the application $\mathcal A$ is a contraction, and therefore has a unique fixed point. 
\end{enumerate}

We then need to check that the solution $(\tx, v_s, u_s)$ of \eqref{EDO-tx}-\eqref{eq:vs}-\eqref{eq:NS-shifted-1-u} provided by the fixed point of $\mathcal A$ is indeed a solution of the original problem. 
Since system \eqref{eq:NS-shifted-1} has been modified, this is not completely obvious. In fact, we need to check that the function $w_s=u_s-\mu \ln v_s$ is indeed equal to $w^0 (x+\tx(t))$. To that end, let us compute the equation satisfied by $w_s$ if $v_s$ is the solution of \eqref{eq:vs} and if $u_s$ is the solution of \eqref{eq:NS-shifted-1-u}.
Combining \eqref{eq:vs} and \eqref{eq:NS-shifted-1-u}, we have
\be\label{eq:ws-para}
\p_t w_s - \tilde x'(t) \p_x w_s -\mu \p_x\left(\frac{1}{v_s}\p_x w_s\right)=-\mu \p_x\left(\frac{1}{v_s}\p_x \tw^0(x+\tx(t))\right).
\ee
Furthermore, the condition $u_{s|x=0^+}=u_-$ ensure that
\[
\tw_{s|x=0^+}= u_-- \mu \p_x \tv_{s |x=0^+},
\]
and using the equation \eqref{eq:vs} together with~\eqref{EDO-tx}
\begin{align*}
\p_x w_{s|x=0^+}
& = \p_x u_{s|x=0^+} - \mu \p_{xx} \ln v_{s|x=0^+} \\
& = \frac{\tw^0(\tx(t))-u_-}{\mu}\tx'(t) + \tx'(t) \p_x v_{s|x=0^+} + \p_x \tw^0(\tx(t)).
\end{align*}
Taking a linear combination of these two equations leads to
\be\label{BC-w}
\mu \p_x w_{s|x=0^+} + \tx'(t) w_{s|x=0^+}
= \tx'(t)\tw^0(\tx(t)) + \mu \p_x \tw^0(\tx(t)).
\ee
It can be easily proved that the solution of \eqref{eq:ws-para}-\eqref{BC-w} endowed with the initial data $\tw^0$ is the  function $\tw^0(x+\tx(t))$.
Thus the function $(v_s,u_s)$ constructed as the solution of \eqref{eq:vs}-\eqref{eq:NS-shifted-1-u}, where $\tx$ is the solution of \eqref{EDO-tx}, is in fact a solution of \eqref{eq:NS-shifted-1}. We extend this solution in $x<0$ by setting $v_s(t,x)=1$, $u_s(t,x)=u_-$, and we set
\[
p_s(t,x)=-\mu\p_x \tu_{s|x=0^+}= \tx'(t)(u_--\tw^0(\tx(t)))\quad \forall x<0.
\]
Eventually, we come back to the original variables and set 
$(v,u,p)(t,x)=(v_s,u_s,p_s)(t,x-\tx(t)).$
Then it is easily checked that $(v,u,p)$ is a solution of the original system \eqref{eq:NS-lim}.

\medskip
\begin{rmk}[About the regularity of $\tx$]
In the above discussion, we have claimed that we will prove the existence of a fixed point $\tx$ in $H^2_{loc}(\R_+)$. 
Let us discuss why this regularity is required on $\tx$.
First, we need a control of $\tx'$ in $L^\infty(\R_+)$ in order to control the transport equation~\eqref{eq:ws} satisfied by $w_s$.
Next, we see in~\eqref{EDO-tx} that the control in $L^\infty$ of $\tx$ requires a bound on $\partial_x u_s$ in $L^\infty(\R_+ \times \R_+)$, while this latter bound would a priori rely on a control of $\tx''$ in $L^2_{loc}(\R_+)$ .
Therefore the regularity $\tx \in H^2_{loc}(\R_+)$ is the minimal regularity which allows us to formally close the fixed point argument with the regularity~\eqref{reg-uv}-\eqref{reg-dtuv} on $u-\bar u$.
\end{rmk}


\section*{Acknowledgement}
This project has received funding from the European Research Council (ERC) under the European Union's Horizon 2020 research and innovation program Grant agreement No 637653, project BLOC ``Mathematical Study of Boundary Layers in Oceanic Motion’’. This work was supported by the SingFlows and CRISIS projects, grants ANR-18-CE40-0027 and ANR-20-CE40-0020-01 of the French National Research Agency (ANR).
A.-L. D. acknowledges the support of the Institut Universitaire de France.
This material is based upon work supported by the National Science Foundation under Grant No. DMS-1928930 while the authors participated in a program hosted by the Mathematical Sciences Research Institute in Berkeley, California, during the Spring 2021 semester.

\bibliography{bibli-DP}
\end{document}